\documentclass[11pt]{amsart}
\usepackage{amsmath}
\usepackage{amsthm}
\usepackage{latexsym}
\usepackage{amssymb}
\usepackage{xspace}
\usepackage{amscd}
\usepackage{enumerate}
\usepackage{lscape}
\usepackage{color}
\usepackage{amsmath}
\usepackage{amsmath, pb-diagram}
\usepackage[all]{xy}
\usepackage{amssymb}
\usepackage{amscd}
\newtheorem{thm}{Theorem}[section]
 \newtheorem{cor}[thm]{Corollary}
 \newtheorem{lem}[thm]{Lemma}
 \newtheorem{prop}[thm]{Proposition}
 \theoremstyle{definition}
 
  \newtheorem{ex}[thm]{Example}
 \newtheorem{rem}[thm]{Remark}
  \newtheorem{rems}[thm]{Remarks}
 \numberwithin{equation}{section}

\DeclareMathOperator{\End}{End}
\DeclareMathOperator{\Hom}{Hom}
\DeclareMathOperator{\Ker}{Ker}
\DeclareMathOperator{\im}{Im}
\DeclareMathOperator{\RMod}{R{\rm -Mod}}
\DeclareMathOperator{\ModR}{{\rm Mod-}R}

\begin{document}
\author{ Pedro A. Guil Asensio, Truong Cong Quynh and Ashish K. Srivastava}
\thanks{First author has been partially supported by the DGI (MTM2010-20940-C02-02). Part of
the sources of this grant come from the FEDER funds of the European Union. The second author has been partially founded by the Vietnam National Foundation for Science and Technology Development (NAFOSTED). The work of the third author is partially supported by a grant from Simons Foundation (grant number 426367).}
\address{Departamento de Mathematicas, Universidad de Murcia, Murcia, 30100, Spain}
\email{paguil@um.es}
\address{Department of Mathematics, Danang University, 459 Ton Duc Thang, Danang city, Vietnam}
\email{tcquynh@dce.udn.vn; tcquynh@live.com}
\address{Department of Mathematics and Statistics, St. Louis University, St. Louis, MO-63103, USA}
\email{asrivas3@slu.edu}
\title[Invariance of modules]{Additive unit structure of endomorphism rings and invariance of modules}
\begin{abstract}
We use the type theory for rings of operators due to Kaplansky to describe the structure of modules that are invariant under automorphisms of their injective envelopes. Also, we highlight the importance of Boolean rings in the study of such modules. As a consequence of this approach, we are able to further the study initiated by Dickson and Fuller regarding when a module invariant under automorphisms of its injective envelope is invariant under any endomorphism of it. In particular, we find conditions for several classes of noetherian rings which ensure that modules invariant under automorphisms of their injective envelopes are quasi-injective. In the case of a commutative noetherian ring, we show that any automorphism-invariant module is quasi-injective. We also provide multiple examples to show that our conditions are the best possible, in the sense that if we relax them further then there exist automorphism-invariant modules which are not quasi-injective. We finish this paper by dualizing our results to the automorphism-coinvariant case.
\end{abstract}

\maketitle

\bigskip

\section{Introduction and notation.}

\bigskip

\noindent The study of rings additively generated by their units has a long tradition (see \cite{D,Wolfson, Zelinsky}) but in recent years there have been substantial advances in this area (see, for example, \cite{GPS, Vamos, VW}). V\'{a}mos in \cite{Vamos} proved that every element of a right self-injective ring $R$ is a sum of two units if $R$ has no non-zero corner ring which is Boolean. This result was extended later in \cite{KS, KS2} by using the type theory for von Neumann regular right self-injective rings developed by Kaplansky. Historically, the theory of types was first proposed by Murray and von Neumann \cite{MV} but it was developed as a classification scheme by Kaplansky in \cite{Kap} for a certain class of rings of operators which are usually called Baer rings. Since von Neumann regular right self-injective rings are Baer rings, Kaplansky's theory is applicable to them. Following this approach, it has been shown in \cite{KS, KS2} that if $R$ is a right self-injective ring such that $R$ has no homomorphic image isomorphic to the field of two elements $\mathbb F_2$, then each element of $R$ is a sum of two units; and that if $R$ has a homomorphic image isomorphic to $\mathbb F_2$ but no homomorphic image isomorphic to $\mathbb F_2 \times \mathbb F_2$ then each element of $R$ is a sum of two or three units.

It has been recently shown in \cite{GKS, GS1}, that the above results can be successfully applied to the classical problem of knowing when a module which is invariant under automorphisms of its injective envelope is, indeed, invariant under any endomorphism of it. The reason is that, if any endomorphism of this injective envelope is a sum of automorphisms, then the invariance of the module under automorphisms of it automatically implies its invariance under any endomorphism.

Let us recall that the study of modules that are invariant under automorphisms of their injective envelopes was initiated in the late sixties by Dickson and Fuller \cite{DF} for the particular case of finite-dimensional algebras over a field. But it has been in the last few years when this notion has been extensively studied over general rings and modules. These modules have been called in the literature automorphism-invariant modules (see \cite{AFT, ESS, GKBS, GS1, GS2, GS3, LZ, SS2}).

On the other hand, Johnson and Wong \cite{JW} proved that a module $M$ is invariant under endomorphisms of its injective envelope if and only if any homomorphism from a submodule $N$ of $M$ to $M$ extends to an endomorphism of $M$. These modules are called quasi-injective. Motivated by this result, Jain and Singh introduced in \cite{JS} pseudo-injective modules as those modules $M$ for which any monomorphism from a submodule of $M$ to $M$ extends to an endomorphism of $M$. And an analogous result to the Johnson and Wong's characterization of quasi-injective modules has been recently proved in \cite{ESS} and \cite{GKBS}, where it is shown that a module $M$ is automorphism-invariant if and only if it is pseudo-injective.

Clearly, any quasi-injective module is automorphism-invariant, but there are several examples of automorphism-invariant modules that are not quasi-injective (see e.g. \cite{ESS, JS, Teply}). It is then natural to ask when an automorphism-invariant module is quasi-injective. It has been shown in \cite{GS2} that if $M$ is a right $R$-module such that $\End_R(M)$ has no homomorphic image isomorphic to $\mathbb{F}_2$, then $M$ is quasi-injective if and only it is automorphism-invariant. In particular, this is the case when $M$ is a module over an $F$-algebra, where $F$ is a field  with more than two elements; thus extending the previously mentioned result of Dickson and Fuller for indecomposable modules \cite{DF}.

This problem of characterizing when an automorphism-invariant module is quasi-injective has been recently continued in \cite{AFT}, where the authors give different partial characterizations of when an automorphism-invariant module $M$ of finite Goldie dimension is quasi-injective and they link this question to the theory of Boolean rings. In fact, V\'{a}mos was the first to highlight the importance of Boolean rings in understanding the additive unit structure of endomorphism rings of injective modules in \cite{Vamos}. The main objective of this paper is to connect these results in \cite{AFT, Vamos} with the fruitful techniques introduced in \cite{GKS} to understand the structure of automorphism-invariant modules in terms of the endomorphism ring of their injective envelopes modulo their Jacobson radical. This allows us to get critical test theorems for automorphism-invariant modules which are possible candidates to be non quasi-injective. We prove then new results that extend, complete and contextualize the different characterizations initiated in \cite{AFT}. We also provide several examples which complement these results, as well as outline the limits of their possible extensions. For example, these examples show that even an indecomposable automorphism invariant finitely generated right module over a finite-dimensional algebra over a field $F$ does not need to be quasi-injective, thus answering in the negative a question posed by Facchini.

Recall that a general theory of modules which are invariant (resp., coinvariant) under automorphisms of their envelopes (resp., covers) has been developed in \cite{GKS}. This theory includes, in particular, modules which are invariant under automorphisms of their injective, pure-injective or cotorsion envelopes, as well as of modules which are coinvariant under automorphisms of their projective or flat covers. We develop our results under this much more general setting and then we apply them to the classical situation of modules invariant (resp., coinvariant) under automorphisms of their injective envelopes (resp., projective covers).

We begin by extending in Section 2 several key observations from \cite{GKS} about the additive unit structure of a right self-injective von Neumann regular ring. These observations will be critical to obtain in Section 3 our characterizations of when an $\mathcal X$-automorphism invariant module is $\mathcal X$-endomorphism invariant. These results are then applied in Section 4 to study when automorphism-invariant modules over certain classes of noetherian or artinian rings are quasi-injective. For example, we characterize right bounded right noetherian rings over which automorphism-invariant modules are quasi-injective. We first show that if $R$ is a commutative noetherian ring such that $R$ has no homomorphic image isomorphic to $\mathbb F_2\times \mathbb F_2$, then any automorphism-invariant $R$-module is quasi-injective and then finally as a consequence, we show that an automorphism-invariant module over any commutative noetherian ring is quasi-injective. Several examples are also included to show that we cannot expect to relax these restrictions any further. We close the paper by dualizing in Section 5 these results for automorphism-coinvariant modules.

Throughout this paper, all rings will be associative rings with identity and our modules will be unitary right modules unless otherwise is stated. We will denote by $J(R)$ the Jacobson radical of a ring $R$ and by $|X|$, the cardinality of a set $X$. A module $M$ is called {\it square-free} if $M$ does not contain a nonzero submodule $N$ isomorphic to $X\oplus X$ for some module $X$. And a ring $R$ is called a Boolean ring if each element of $R$ is an idempotent. We will say that a ring $R$ is a {\it semiboolean ring} if $R/J(R)$ is a Boolean ring; and we will say that a module $M$ is {\it semiboolean} if its endomorphism ring is a semiboolean ring. Note that it has been proved in \cite{GKS} that if $M$ is an $\mathcal X$-automorphism invariant (resp., $\mathcal X$-automorphism coinvariant) module with a monomorphic $\mathcal X$-envelope $u:M\rightarrow X$ (resp., epimorphic cover $p:X\rightarrow M$) such that $\End(X)/J(\End(X))$ is a von Neumann regular right self-injective ring and idempotents lift modulo $J(\End(X))$, then idempotents in $\End_R(M)/J(\End_R(M))$ lift modulo $J(\End(M))$ and therefore, for endomorphism rings of such classes of modules our definition of semiboolean ring coincides with that of \cite{NZ}. We refer to \cite{AF, CLVW, F, MM} for any undefined notion used along the text.

\bigskip

\section{Observations on the structure of von Neumann regular right self-injective rings.}

\bigskip

\noindent We begin with some observations on the structure of von Neumann regular right self-injective rings which will help us in refining  the structure of endomorphism rings of $\mathcal X$-automorphism invariant and  $\mathcal X$-automorphism coinvariant modules. It is well known that any von Neumann regular right self-injective ring $R$ can be decomposed as a product $R=R_1 \times R_2 \times R_3 \times R_4 \times R_5$ where $R_1$ is of type $I_f$, $R_2$ is of type $I_\infty$, $R_3$ is of type $II_f$, $R_4$ is of type $II_\infty$, and $R_5$ is of type $III$. Using ideas that were implicit in \cite{KS}, it was shown in \cite{GKS} that this ring decomposition can be streamlined further and shown that if $R$ is a von Neumann regular right self-injective ring then $R=R_1\times R_2$, where $R_1$ is an abelian regular self-injective ring and each element of the ring $R_2$ is a sum of two units. Furthermore, it is known that if $R$ is an abelian regular self-injective ring then $R=R_1\times R_2$, where $R_1$ is a Boolean ring and each element of ring $R_2$ is a sum of two units \cite{KS2}. Thus, combining them together, we obtain:

\begin{thm}
Let $R$ be a von Neumann regular right self-injective ring. Then $R=R_1\times R_2$, where $R_1$ is a Boolean ring and each element of the ring $R_2$ is a sum of two units.
\end{thm}

Using this theorem, the arguments used in the proof of Theorem 2.4 in \cite{GKS} can be suitably modified to obtain the following theorem. As its proof is similar to that of  \cite[Theorem~2.4]{GKS}, we will only give a brief sketch of it.

\begin{thm} \label{type}
Let $S$ be a von Neumann regular right self-injective ring and $R$, a subring of $S$ which is stable under left multiplication by units of $S$. Then $R$ is a von Neumann regular ring and $R=R_1\times R_2$, where $R_1$ is a Boolean ring and $R_2$ is a von Neumann regular right self-injective ring.
\end{thm}

\begin{proof}
By the above theorem, we know that $S=S_1\times S_2$, where $S_1$ is a Boolean ring and each element of the ring $S_2$ is a sum of two units. As $R$ is a subring of $S$, we may view any element $a$ of $R$ as $a=a_1\times a_2$ where $a_1\in S_1$ and
$a_2\in S_2$. Since any element $s_2\in S_2$ is the sum of two units, say $s_2=t_2+t'_2$, we may write the element
$0\times s_2 = 1_{S_1}\times t_2 + (-1_{S_2})\times t'_2$
as the sum of two units of $S$ and, as $R$ is stable under left multiplication by units in $S$, this means that $0\times s_2\in R$ for any $s_2\in S_2$. Call $R_2=S_2$ and define
$$R_1=\{s_1\in S_1 : \exists s_2\in S_2 \text{ such that }s_1\times s_2\in R \}.$$

Then any $s_1\times 0$ with $s_1\in R_1$ can be written as $s_1\times 0=s_1\times s_2 - 0\times s_2$ with $s_1\times s_2\in R$. Therefore, $s_1\times 0\in R$ for any $s_1\in R_1$ and we deduce that $R=R_1\times R_2$. This gives a decomposition for $R=R_1\times R_2$ with $R_2=S_2$ and therefore $R_2$ is a von Neumann regular right self-injective ring. As every element of $R_1$ is in $S_1$, it follows that each element of $R_1$ is an idempotent and hence $R_1$ is a Boolean ring.
\end{proof}

\bigskip

\section{When is an $\mathcal X$-automorphism invariant module $\mathcal X$-endomorphism invariant.}

\noindent We begin this section by recalling some notation and definitions from \cite{GKS} that we will need along this paper.  Recall that a class $\mathcal{X}$ of right modules closed under isomorphisms is called an {\em enveloping class} if for any module $M$ there exists a universal homomorphism with respect to $\mathcal X$, $u:M\rightarrow X(M)$, with $X(M)\in\mathcal{X}$, in the sense that any other morphism from $M$ to a module in $\mathcal X$ factors through $u$ and,
moreover, $u$ is minimal in the sense that whenever $u$ has a factorization $u=h\circ u$, then $h$ must be an automorphism. This morphism $u:M\rightarrow X(M)$ is called the {\em $\mathcal X$-envelope} of $M$. And this envelope is called a {\em monomorphic envelope} if, in addition, $u$ is a monomorphism.
The dual notions of covering class, $\mathcal X$-cover and epimorphic $\mathcal X$-cover are obtained by reversing arrows.
Then, a module $M$ having an $\mathcal X$-envelope (resp., $\mathcal X$-cover)
is called $\mathcal X$-automorphism invariant (resp., $\mathcal X$-automorphism coinvariant) when it is invariant under any automorphism of its $\mathcal X$-envelope (resp., $\mathcal X$-cover). And $M$ is called $\mathcal X$-endomorphism invariant (resp., $\mathcal X$-endomorphism coinvariant) when it is invariant under any endomorphism of its $\mathcal X$-envelope (resp., $\mathcal X$-cover) (see \cite{GKS}). Finally, when $\mathcal X$ is the class of injective modules, we will simply talk of automorphism-invariant and quasi-injective modules, instead of $\mathcal{X}$-automorphism invariant and $\mathcal X$-endomorphism invariant modules. And if $\mathcal X$ is the class of projective modules, we will talk of automorphism coinvariant and quasi-projective modules instead of $\mathcal X$-automorphism coinvariant and $\mathcal X$-endomorphism coinvariant modules.

 It is clear from the definition that any $\mathcal X$-endomorphism invariant module is $\mathcal X$-automorphism invariant. We give below some examples that show that the converse is no longer true, even for indecomposable $\mathcal X$-automorphism invariant modules having finite Goldie dimension.

\begin{ex}\label{noqi} $($\cite{SS2}$)$
 Let $R=\left[
\begin{array}{ccc}
\mathbb F_2 & \mathbb F_2 & \mathbb F_2 \\
0 & \mathbb F_2 & 0 \\
0 & 0 & \mathbb F_2  \\
\end{array}
\right] $ where $\mathbb F_2$ is the field of two elements.

Take $M=\left[
\begin{array}{ccc}
\mathbb F_2 & \mathbb F_2 & \mathbb F_2 \\
0 & 0 & 0 \\
0 & 0 & 0  \\
\end{array}
\right] $. As $M= e_{11}R$, where $e_{11}$ is a primitive idempotent, $M$ is an indecomposable right $R$-module. Since $R$ is a finite-dimensional $\mathbb F_2$-algebra, $M$ is an artinian right $R$-module and hence it has finite Goldie dimension. Note that $M$ has two simple submodules $S_1=e_{12}R=\left[
\begin{array}{ccc}
0 & \mathbb F_2 & 0 \\
0 & 0 & 0 \\
0 & 0 &0  \\
\end{array}
\right] $ and $S_2=e_{13}R=\left[
\begin{array}{ccc}
0 & 0 & \mathbb F_2 \\
0 & 0 & 0 \\
0 & 0 & 0  \\
\end{array}
\right]$. Let $\mathcal X$ be the class of injective right $R$-modules. It may be checked that the only automorphism of the injective envelope of $M$ is the identity. Therefore, $M$ is automorphism invariant. But clearly $M$ is not quasi-injective as it is not uniform. Thus, we have an example of an indecomposable module with finite Goldie dimension which is automorphism invariant but not quasi-injective.
\end{ex}

\begin{ex}\label{noqi2} (Teply, see \cite{JS})
Let $A=\mathbb F_2[x]$ and $$R=\left[
\begin{array}{cc}
A/(x) & 0 \\
A/(x) & A/(x^2) \\
\end{array}
\right] $$

Let $M=\left[
\begin{array}{cc}
0 & 0 \\
A/(x) & A/(x^2) \\
\end{array}
\right]$. As $M=e_{22}R$, where $e_{22}$ is a primitive idempotent, $M$ is an indecomposable right $R$-module. Note that $M$ has two simple submodules $S_1=\left[
\begin{array}{cc}
0 & 0 \\
A/(x) & 0 \\
\end{array}
\right]$ and $S_2=\left[
\begin{array}{cc}
0 & 0 \\
0 & (x)/(x^2) \\
\end{array}
\right]$ such that $S_1\oplus S_2$ is essential in $M$.

Clearly, $R$ is a finite-dimensional $\mathbb{F}_2$-algebra. Let $\mathcal X$ be the class of injective right $R$-modules. Teply proved in \cite{JS} that $M$ is automorphism invariant. But $M$ is not quasi-injective as $M$ is not uniform. This gives another example of an indecomposable module with finite Goldie dimension which is automorphism invariant but not quasi-injective.
\end{ex}

Let $\mathcal{X}$ be an enveloping (resp., covering) class of right $R$-modules. If $u:M\rightarrow X$ is a monomorphic $\mathcal{X}$-envelope (resp., $p:X\rightarrow M$ is an epimorphic cover) of a module $M_R$ which is $\mathcal X$-automorphism invariant (resp., coinvariant) and $\End(X)/J(\End(X))$ is a von Neumann regular right self-injective ring and idempotents lift modulo $J(\End(X))$ then, as shown in \cite{GKS}, there exists an injective
ring homomorphism
$$
\Psi: \End(M)/J(\End(M))\longrightarrow \End(X)/J(\End(X))
$$
given by the
rule $\Psi(f+J(\End(M)))=g+J(\End(X))$ where $g\in \End(X)$ such that $g\circ u=u\circ f$ (resp., $p\circ g=f\circ p$). This allows us to identify $\End(M)/J(\End(M))$ with the subring $\im
\Psi\subseteq \End(X)/J(\End(X)).$

In view of Theorem \ref{type}, the decomposition theorem for $\mathcal{X}$-automorphism invariant (resp., $\mathcal{X}$-automorphism coinvariant) modules obtained in \cite[Theorem 3.12, Theorem 4.9]{GKS} may be modified and stated as follows.

\begin{thm} \label{struct}
Let $\mathcal{X}$ be an enveloping (resp., covering) class of right $R$-modules. Let $u:M\rightarrow X$ be a monomorphic $\mathcal{X}$-envelope (resp., $p:X\rightarrow M$ is an epimorphic cover) of a module $M_R$ such that $M$ is $\mathcal{X}$-automorphism invariant (resp., $\mathcal{X}$-automorphism coinvariant) and $\End(X)/J(\End(X))$ is a von Neumann regular right self-injective ring and idempotents lift modulo $J(\End(X))$.

Then $\End(M)/J(\End(M))$ is also a von Neumann regular ring and idempotents in $\End(M)/J(\End(M))$ lift to idempotents in $\End(M)$.

\noindent Moreover, $M$ admits a decomposition $M=N\oplus L$ such that:
\begin{enumerate}
\item[(i)] $N$ is a semiboolean module.
\item[(ii)] $L$ is $\mathcal X$-endomorphism invariant (resp., $\mathcal X$-endomorphism coinvariant).
\end{enumerate}
In particular, $\End(M)/J(\End(M))$ is the direct product of a Boolean ring and a right self-injective von Neumann regular ring.
\end{thm}

Note that the above theorem is new even when $\mathcal X$ is the class of injective modules. In this case, it properly extends \cite[Proposition 3.15]{AFT} as in that proposition, the above decomposition has been obtained for a square-free automorphism-invariant module of finite Goldie dimension.

Using the additive unit structure of von Neumann regular right self-injective rings, a condition is given in \cite{GKS} under which $\mathcal X$-automorphism invariant modules are $\mathcal X$-endomorphism invariant.

\begin{thm} \label{noF2}\cite{GKS}
Let $M_R$ be an $\mathcal X$-automorphism invariant module with a monomorphic $\mathcal X$-envelope $u:M\rightarrow X$ such that $\End(X)/J(\End(X))$ is von Neumann regular right self-injective and idempotents lift modulo $J(\End(X))$. If $\End(M)$ has no homomorphic image isomorphic to $\mathbb F_2$, then $M$ is $\mathcal X$-endomorphism invariant.
\end{thm}

When $R$ is a commutative ring, the above theorem has the following consequence.

\begin{cor} \label{comm}
If $R$ is a commutative ring with no homomorphic image isomorphic to $\mathbb F_2$ and $M$ is an $\mathcal X$-automorphism invariant $R$-module with a monomorphic $\mathcal X$-envelope $u:M\rightarrow X$ such that $\End(X)/J(\End(X))$ is von Neumann regular right self-injective and idempotents lift modulo $J(\End(X))$, then $M$ is $\mathcal X$-endomorphism invariant.
\end{cor}

\begin{proof}
If $R$ is a commutative ring then there exists a ring homomorphism $f:R\rightarrow \End_R(M)$. Now, if there exists a ring homomorphism $g:\End(M)\rightarrow \mathbb F_2$, then the composition $g\circ f:R\rightarrow \mathbb F_2$ gives a ring homomorphism, thus contradicting the assumption that $R$ has no homomorphic image isomorphic to $\mathbb F_2$. This means that $\End(M)$ has no homomorphic image isomorphic to $\mathbb F_2$, and hence, $M$ is $\mathcal X$-endomorphism invariant by the above theorem.
\end{proof}

We proceed now to characterize when an indecomposable $\mathcal X$-automorphism invariant module is $\mathcal X$-endomorphism invariant. This characterization will play a central role in our applications in the next section.

\begin{lem} \label{indecomp-1}
Let $M$ be an $\mathcal X$-automorphism invariant module with a monomorphic $\mathcal X$-envelope $u:M\rightarrow X$ such that $\End(X)/J(\End(X))$ is von Neumann regular right self-injective and idempotents lift modulo $J(\End(X))$. If $X$ is indecomposable, then $M$ is $\mathcal X$-endomorphism invariant.
\end{lem}

\begin{proof}
Assume $X$ is indecomposable. Then $\End(X)/J(\End(X))$ is a von Neumann regular ring with no non-trivial idempotents and consequently, $\End(X)/J(\End(X))$ is a division ring. Thus each element of $\End(X)/J(\End(X))$, and hence of $\End(X)$, is a sum of two or three units \cite{KS2}. Therefore, $M$ is invariant under any endomorphism of $X$.
\end{proof}

Our next proposition shows that the converse of the above lemma holds under the additional assumption that $M$ is an indecomposable module.

\begin{prop} \label{indecomp}
Let $M$ be an indecomposable $\mathcal X$-automorphism invariant module with a monomorphic $\mathcal X$-envelope $u:M\rightarrow X$ such that $\End(X)/J(\End(X))$ is von Neumann regular right self-injective and idempotents lift modulo $J(\End(X))$. Then the following statements are equivalent:
\begin{enumerate}
\item $M$ is $\mathcal X$-endomorphism invariant.
\item $X$ is indecomposable.
\end{enumerate}
\end{prop}

\begin{proof}
Assume $M$ is $\mathcal X$-endomorphism invariant. Then, in particular, $M$ is invariant under any idempotent endomorphism of $X$. As $M$ is indecomposable, this means that the only idempotents in $\End(X)$ are $0$ and $1$. So $X$ is indecomposable. The reverse implication follows from the above lemma.
\end{proof}

As a consequence, we have the following characterization of indecomposable $\mathcal X$-automorphism invariant modules which are not $\mathcal X$-endomorphism invariant.

\begin{thm} \label{noei-struct}
Let $M$ be an indecomposable $\mathcal X$-automorphism invariant module with a monomorphic $\mathcal X$-envelope $u:M\rightarrow X$ such that $\End(X)/J(\End(X))$ is von Neumann regular right self-injective and idempotents lift modulo $J(\End(X))$. Assume that $M$ is not $\mathcal X$-endomorphism invariant. Then $\End(M)/J(\End(M))\cong \mathbb F_2$ and $\End(X)/J(\End(X))$ has a homomorphic image isomorphic to $\mathbb F_2 \times \mathbb F_2$.
\end{thm}

\begin{proof}
By Theorem \ref{struct}, $\End(M)/J(\End(M))$ is a von Neumann regular ring. As $M$ is indecomposable, it follows that $\End(M)/J(\End(M))$ is a von Neumann regular ring with no non-trivial idempotents and consequently, $\End(M)/J(\End(M))$ is a division ring. Now, by Theorem \ref{noF2}, we know that $\End(M)$ has a homomorphic image isomorphic to $\mathbb F_2$. Thus, we have $\End(M)/J(\End(M))\cong \mathbb F_2$. By \cite{KS2} it is known that if $\End(X)/J(\End(X))$ has no homomorphic image isomorphic to $\mathbb F_2 \times \mathbb F_2$, then each element of $\End(X)$ is a sum of two or three units. This would make $M$ an $\mathcal X$-endomorphism invariant module, a contradiction to our assumption. This shows that $\End(X)/J(\End(X))$ has a homomorphic image isomorphic to $\mathbb F_2 \times \mathbb F_2$.
\end{proof}

\section{Applications to Automorphism-invariant modules.}

\bigskip

\noindent In this section, we will apply our previously obtained results to the setting in which $\mathcal X$ is the class of injective modules. In this case, $\mathcal X$-automorphism invariant modules are just called automorphism-invariant modules and $\mathcal X$-endomorphism invariant modules are called quasi-injective modules. We will denote the injective envelope of a module $M$ by $E(M)$.

In \cite{AFT}, several equivalent characterizations are given of when an  automorphism-invariant module with finite Goldie dimension is quasi-injective. Let us note that if $M$ is an automorphism-invariant module of finite Goldie dimension, then $M=\oplus_{i=1}^n M_i$ is a direct sum of indecomposable modules also having finite Goldie dimension. Moreover, it is known that a finite direct sum of automorphism-invariant modules is again automorphism-invariant if and only if the direct summands are relatively injective \cite{LZ}. Thus, the module $M$ is quasi-injective if and only if so is each $M_i$. This means that the question of whether an automorphism-invariant module with finite Goldie dimension is quasi-injective reduces to study when an indecomposable automorphism-invariant module having finite Goldie dimension is quasi-injective. Apparently, the problem of whether these modules are always quasi-injective has been open until now in the literature \cite{AFT}. Note however, that Examples~\ref{noqi} and \ref{noqi2} in the previous section provide negative answers to this question.

We will begin our characterization by proving the following theorem.

\begin{thm}\label{thm:main} Let $M$ be an indecomposable automorphism-invariant module with finite Goldie dimension such that $M$ is not quasi-injective. Then
\begin{enumerate}
\item $\End(M)/J(\End(M))\cong \mathbb{F}_2$.
\item There exists a finite set of non-isomorphic indecomposable injective modules $\{E_i\}_{i=1}^n$, with $n\geq 2$, such that $E(M)=\oplus_{i=1}^n E_i$ and $\End(E_i)/J(\End(E_i))\cong \mathbb{F}_2$ for every $i=1,\ldots,n$.
\end{enumerate}
\end{thm}

\begin{proof} By Proposition \ref{noei-struct}, $\End(M)/J(\End(M))\cong \mathbb{F}_2$.
On the other hand, we know from Theorem~\ref{struct} that $M=N\oplus L$ is the direct sum of a quasi-injective module $L$ and a semiboolean module $N$. As $M$ is indecomposable and non quasi-injective, this means that $M=N$. Call $E=E(M)$ the injective envelope of $M$. Since  $M$ has finite Goldie dimension, $E$ must be a finite direct sum of indecomposable injective modules. But again, the proof of Theorem~\ref{struct} shows that $\End(E)/J(\End(E))$ is a Boolean ring and therefore, $E$ is square free. So
$E= E_1\oplus\ldots\oplus E_n$ for some indecomposable injective modules $E_i$ satisfying that $E_i$ is  non-isomorphic to $E_j$ for any $j\neq i$. Moreover, Theorem~\ref{noei-struct} shows that $n\geq 2$. Finally, as each $E_i$ is an indecomposable injective module, $\End(E_i)/J(\End(E_i))$ is a division ring for every $i$. And, as $\End(E)/J(\End(E))$ is a Boolean ring, this means that $\End(E_i)/J(\End(E_i))\cong \mathbb{F}_2$ for each $i=1,\ldots,n$.
\end{proof}

In particular, if $M$ is a finitely cogenerated module, its socle ${\rm Soc}(M)$ is finitely generated and essential in $M$. So we get the following corollary.

\begin{cor}\label{cor:sim} Let $M$ be an indecomposable finitely cogenerated automorphism-invariant module which is not quasi-injective. If we write ${\rm Soc}(M)=\oplus_{i=1}^n C_i$ as a direct sum of indecomposable modules, then
\begin{enumerate}
\item $n\geq 2$.
\item $\End(M)/J(\End(M))\cong \mathbb{F}_2.$
\item $\End(C_i)\cong \mathbb{F}_2$ for every $i\in I$.
\item $C_i\ncong C_j$ if $i\neq j$.
\end{enumerate}
\end{cor}

\begin{proof}
Let us note that $E(M)=\oplus_{i=1}^nE(C_i)$. Therefore, the result is an immediate consequence of the above theorem.
\end{proof}

\noindent Our next step will be to extend the above corollary to automorphism-invariant modules over right bounded right noetherian rings. Recall that a ring $R$ is called {\it right bounded} if each essential right ideal of $R$ contains a two-sided ideal which is essential as a right ideal. A right noetherian ring $R$ is right bounded if and only if each essential right ideal of $R$ contains a non-zero two-sided ideal \cite{Jategaonkar}. Moreover, a ring $R$ is called a {\it right FBN ring} if $R$ is a right noetherian ring such that $R/P$ is right bounded for each prime ideal $P$ of $R$. It is well known that over a right bounded right noetherian ring $R$, any indecomposable injective right module is of the form $E(U_R)$, where $U_R$ is a uniform right submodule of $R/P$ for some prime ideal $P$. Furthermore, if $R$ is right FBN, and  $U_R,U'_R$ are two non-zero uniform right submodules of $R/P$, then $E(U)\cong E(U')$  \cite[p. 163]{GW}. Our next result gives a sufficient condition for an automorphism-invariant module over a right bounded right noetherian ring to be quasi-injective. We begin by proving the following technical proposition.

\begin{prop} \label{fbn}
Let $R$ be a right bounded right noetherian ring and $M$, a non quasi-injective, semiboolean, automorphism-invariant module. Then there exists a set $\{C_i\}_{i\in I}$ of non-isomorphic simple right $R$-modules with $|I|\geq 2$ such that ${\rm Soc}(M)=\oplus_{i\in I}C_i$ is essential in $M$ and $\End(C_i)\cong {\mathbb F}_2$ for every $i\in I$. 
Moreover, if $R$ is a commutative ring, then $|C_i|=2$  for every $i\in I$.
\end{prop}

\begin{proof}
Let us first note that as $R$ is a right bounded ring, there exist prime ideals $\{P_i\}_{i\in I}$ and non-zero uniform right ideals $\{U_i\}_{i\in I}$ such that $E(M)=\oplus_{i\in I}E(U_i)$ and each $U_i$ is a submodule of $R/P_i$. Moreover, $E(M)$ is also a semiboolean module by the proof of Theorem~\ref{struct} and thus, $\End(E(M))/J(\End(E(M)))$ is square-free. It follows then that $E(U_i)\ncong E(U_j)$ if $i\neq j$. Furthermore, each $\End(E(U_i))/J(\End(E(U_i)))$ is a division ring. Thus $\End(E(U_i))/J(\End(E(U_i)))\cong \mathbb{F}_2$ for every $i\in I$, again because $\End(E(M))/J(\End(E(M)))$ is a Boolean ring. Finally, as $\End(E(M))$ must have a homomorphic image isomorphic to $\mathbb{F}_2\times \mathbb{F}_2$  by Theorem~\ref{noei-struct}, we deduce that $|I|\geq 2$.

Fix now an index $i\in I$ and call $P=P_i,U=U_i, Q=E(U)$ and $T=R/P$. Let us choose a nonzero right ideal $V\subseteq U$. Call $L=\{v\in V|\, r_T(v)\cap V\neq 0\}$, where $r_T(v)$ denotes the right annihilator of $v$ in $T$. 

We first show that any element in $L$ is nilpotent. Let us choose a non-zero element $v\in V$.

We have an ascending chain of right ideals 
$$r_T(v) \subseteq r_T(v^2) \subseteq \ldots \subseteq r_T(v^m) \subseteq \ldots$$ 
So there exists an $m_0$ such that $r_T(v^{m_0})=r_T(v^{m_0+m})$ for each $m\in \mathbb N$, since $R$ is right noetherian. Assume that $v^{m_0} \neq 0$. As $v^{m_0}T\subseteq V$ and
$0\neq r_T(v)\cap V$ is an essential submodule of $V$, since $V$ is uniform,  there exists a $t\in T$ such that $0\neq v^{m_0}t\in r_T(v)$. And thus, $v^{m_0+1}t=0$. But, as $r_T(v^{m_0})=r_T(v^{m_0+1})$, this means that $v^{m_0}t=0$, a contradiction that shows that any element in $L$ is nilpotent.

Therefore, $L$ is a nil subset of $T=R/P$. Now, it is easy to check that $L$ is multiplicatively closed. So $L$ is nilpotent by 
\cite[Theorem 6.21]{GW}. We claim that $L^2=0$. Let $m_0$ the biggest integer such that there exist elements 
$v_1,\ldots,v_{m_0}\in L$ such that $v_{m_0}\cdot\ldots\cdot v_1\neq 0$ and assume on the contrary that $m_0\geq 2$. 
Then, for any $t\in T$, we have that $v_1tv_2\in L$ and thus, $v_{m_0}\cdot\ldots\cdot (v_1tv_2) v_1$ 
is the product of $m_0+1$ elements in $L$ and so it is $0$. It follows that $v_{m_0}\cdot\ldots\cdot v_1Tv_2\cdot v_1=0$ and thus, $v_{m_0}\cdot\ldots\cdot v_1=0$ or $v_2 v_1=0$ since $T$ is a prime ring. 
In any of both cases, we deduce that $v_{m_0}\cdot\ldots\cdot v_1=0$ since $m_0\geq 2$. A contradiction that proves our claim.

So we have that $L^2=0$. Let us distinguish two possibilities. 

\medskip

 \noindent {\bf Case 1.} If $L$ is a right ideal of $T$, then $vT\subseteq L$ for any $v\in L$ and thus, $vTv=0$ as $L^2=0$. Again, as $T$ is prime, we deduce that $v=0$ and thus, $L=0$. But then, the left multiplication $f_v: V\rightarrow V$ by $v$ is a monomorphism for every non-zero $v\in V$ that extends to an isomorphism $g_v:E(V)\rightarrow E(V)$. We claim that $V$ has only two elements. Let us choose two non-zero elements $v,v'\in V$ and assume that they are not equal. Then $g_v,g_{v'}$ are two automorphisms of $E(V)$ and thus, $g_v-g_{v'}\in J(\End(E(V)))$ since $\End(E(V))/J(\End(E(V))\cong \mathbb{F}_2$. Therefore, $g_v-g_{v'}$ has essential kernel and so, $v-v'\in L=0$ and we get that $v=v'$. This means that $V$ has only two elements and thus, it is a simple submodule of $T$ satisfying that $Q=E(V)$.

\medskip

\noindent {\bf Case 2.} Assume that $L$ is not a right ideal of $T$. This means that there exists a $v_0\in L$ and a $t_0\in T$ such that $v_0t_0\notin L$. Therefore, $r_T(v_0t_0)\cap V=0$ and this means that $r_T((v_0t_0)^m)\cap V=0$ for each $m\geq 1$, as $V$ is uniform. Call $v_1=v_0t_0$ and let $f_{v_1^m}: V\rightarrow V$ be the left multiplication by $v_1^m$. Then $f_{v_1^m}$ is a monomorphism for each $m\geq 1$ and reasoning as in Case 1, we deduce that $v_1^m-v_1^{m'}\in L$ for each $m,m'\geq 1$. In particular, $v:=v_1-v_1^3\in L$. But then, $v_1v_0-v_1^3v_0=vv_0=0$ since $v,v_0\in L$. And multiplying on the right by $t_0$, we get that $e=v_1^2$ is a non-zero idempotent of $V$. As $U$ was uniform, we deduce that $U=V$ is a direct summand of $T$.
\medskip

Therefore, either we are in Case 1 for some nonzero submodule V of $W$ and we get that $Q$ is the injective envelope of the simple module $V$ having only two elements or, otherwise, we get that any nonzero submodule of $U$ equals $U$. Therefore, $U$ is a simple right ideal of $T$.

Finally, note that if $R$ is a commutative ring, then $L$ is a right ideal of $T$ and thus, we are always in Case 1. This completes the proof. 
\end{proof}

\begin{rem}
Assume that the ring $R$ in the above proposition is right FBN. Then, for any simple right $R$-module $C$, we have that $r_R(C)$ is a maximal two-sided ideal of $R$ and $R/r_R(C)$ is a simple artinian ring (see \cite[Theorem~9.10]{GW}). Moreover, $R/r_R(C)\cong C_R^n$ for some $n\geq	1$ as a right $R$-module. Therefore, 
$$R/r_R(C)\cong \End_R(R/r_R(C))\cong \End_R(C^n)\cong M_n(\mathbb{F}_2)$$
This means that, under this identification, $C$ becomes a simple right ideal of the full matrix ring $M_n(\mathbb{F}_2)$ and thus, $|C|=2^n$, where $n$ is the Goldie dimension of $R/r_R(C)$. 
\end{rem}

In particular, when $M$ is an indecomposable module, we get

\begin{cor}\label{indec-FBN} 
Let $R$ be a right bounded right noetherian ring and $M$, an indecomposable automorphism-invariant right $R$-module which is not quasi-injective. Then there exists a set $\{C_i\}_{i\in I}$ of non-isomorphic simple right $R$-modules with $|I|\geq 2$ such that ${\rm Soc}(M)=\oplus_{i\in I}C_i$ is essential in $M$ and $\End(C_i)\cong \mathbb{F}_2$ for each $i\in I$.
\end{cor}

\begin{proof}
We know from Theorem~\ref{struct} that $M=N\oplus L$, where $N$ is a semiboolean module, any element in $\End(E(L))/J(\End(E(L)))$ is the sum of two units and moreover, $\End(E(M)/J(\End(E(M))=\End(E(N)/J(\End(E(N))\times \End(E(L)/J(\End(E(L))$. As we are assuming that $M$ is indecomposable and not quasi-injective, we must have $N=M$. So the result follows from the above proposition.
\end{proof}

Another consequence of Proposition~\ref{fbn} is the following corollary.

\begin{cor} \label{comm-nonF2}
Let $R$ be a commutative noetherian ring with no homomorphic images isomorphic to $\mathbb{F}_2\times \mathbb{F}_2$. Then any automorphism-invariant $R$-module is quasi-injective.
\end{cor}

\begin{proof}
Assume on the contrary that $M$ is an automorphism-invariant $R$-module which is not quasi-injective. Again we know from Theorem~\ref{struct} that $M=N\oplus L$, where $N$ is a semiboolean module, any element in $\End(E(L))/J(\End(E(L)))$ is the sum of two units and $\End(E(M)/J(\End(E(M))=\End(E(N)/J(\End(E(N))\times \End(E(L)/J(\End(E(L))$. So $N$ must be also an automorphism-invariant module that is not quasi-injective. Applying now Proposition~\ref{fbn}, we deduce that there exists a set $\{C_i\}_{i\in I}$ of non-isomorphic simple right $R$-modules with $|I|\geq 2$ such that ${\rm Soc}(N)=\oplus_{i\in I}C_i$ is essential in $N$ and $C_i$ has only two elements for every $i\in I$.

Write $C_i=\{0, c_i\}$ for every $i\in I$. Let $p_i:R\rightarrow C_i$ be the homomorphism defined by $p_i(1)=c_i$ and  call $N_i=\Ker(p_i)$. Then $N_i$ is a maximal ideal of $R$. 

Finally, as $|I|\geq 2$, there exist at least two different indexes $i_1,i_2\in I$. Define then $p:R\rightarrow R/N_{i_1}\times R/N_{i_2}$ by $p(1)=p_{i_1}(1)\times p_{i_2}(1)$. Clearly $p$ is a ring homomorphism and it is easy to check that $R/N_{i_1}\times R/N_{i_2}$ is isomorphic to the ring $\mathbb{F}_2\times \mathbb{F}_2$. Let us check that $p$ is surjective.
As  $N_{i_1}, N_{i_2}$ are different maximal deals, there exist elements $r_{i_1}\in N_{i_1}\setminus N_{i_2}$  and $r_{i_2}\in N_{i_2}\setminus N_{i_1}$. And this means that $p(r_{i_1})=c_{i_1}\times 0$ and $p(r_{i_2})=0\times c_{i_2}$. Therefore, $p$ is surjective. This yields a contradiction to our assumption that $R$ has no homomorphic images isomorphic to $\mathbb{F}_2\times \mathbb{F}_2$. Thus, it follows that any automorphism-invariant $R$-module must be quasi-injective.
\end{proof}

In particular, this yields

\begin{cor} \label{local}
Over a commutative noetherian local ring, any automorphism-invariant module is quasi-injective.
\end{cor}

\begin{proof}
If $R$ has a homomorphic image isomorphic to $\mathbb{F}_2\times \mathbb{F}_2$, then it has at last two distinct maximal ideals and so it cannot be local. This shows that if $R$ is a commutative noetherian local ring, then it has no homomorphic images isomorphic to $\mathbb{F}_2\times \mathbb{F}_2$ and consequently, any automorphism-invariant $R$-module is quasi-injective. 
\end{proof}

\begin{rem}
Let us note that any finite-dimensional algebra over a field $K$ is a left and right FBN ring by \cite[Proposition~9.1(a)]{GW}. Therefore, Examples~\ref{noqi} and \ref{noqi2} show that a finitely generated automorphism-invariant module over a (two-sided) FBN ring $R$ does not need to be quasi-injective.  
\end{rem}

Our next step will be to show that any automorphism invariant module over a commutative noetherian ring $R$ is quasi-injective. In order to prove it, we will first show that being automorphism-invariant is a local property for modules over commutative noetherian rings. 

\begin{prop}\label{local-ai}
Let $R$ be a commutative noetherian ring and $M,$ an $R$-module. The following are equivalent.
\begin{enumerate}
\item $M$ is automorphism-invariant.
\item $M_{\mathfrak p}$ is an automorphism-invariant $R_{\mathfrak p}$-module for every prime ideal ${\mathfrak p}$.
\item $M_{\mathfrak m}$ is an automorphism-invariant $R_{\mathfrak m}$-module for every maximal ideal ${\mathfrak m}$.
\end{enumerate}
\end{prop}

\begin{proof}
$(1) \Rightarrow (2)$. Let $\mathfrak p$ be a prime ideal and assume that $M$ is an automorphism-invariant module. Call $E=E(M)$. As the torsion submodules in $R$-Mod for the multiplicative set $S_{\mathfrak p}=R\setminus{\mathfrak p}$ are closed under injective envelopes (see e.g. \cite[Proposition 4.5 (i)]{St}), we get that $E=t(E)\oplus E'$, where $t(E)$ is the torsion submodule of $E$ and $E'=E(M_{\mathfrak p})$ is the injective envelope of $M_{\mathfrak p}$. Let now $f:E'\rightarrow E'$ be an automorphism of $E'$. 
Then the diagonal homomorphism $1_{t(M)}\oplus f: E\rightarrow E$ is an isomorphism of $E$ and thus, $(1_{t(M)}\oplus f)(M)\subseteq M$ by hypothesis. 
And this means that $f(M_{\mathfrak p})=(1_{t(M)}\oplus f)_{\mathfrak p}(M_{\mathfrak p})\subseteq M_{\mathfrak p}$. So $M_{\mathfrak p}$ is automorphism-invariant.
\medskip

$(2) \Rightarrow (3)$. This is trivial.
\medskip

$(3) \Rightarrow (1)$. Let now fix an automorphism $f:E(M)\rightarrow E(M)$ and choose a maximal ideal ${\mathfrak m}$ of $R$. The localization $f_{\mathfrak m}$ of $f$ at $\mathfrak m$ is also an isomorphism.
And localizing at ${\mathfrak m}$ the short exact sequence
\[M\stackrel{f|_M}\to M+f(M)\to (M+f(M))/f(M)\to 0,\]
we obtain the short exact sequence 
\[M_{\mathfrak{m}}\stackrel{f_{\mathfrak m}|_{M_{\mathfrak{m}}}}\to M_{\mathfrak{m}}+f(M)_{\mathfrak{m}}=M_{\mathfrak{m}}+f_{\mathfrak{m}}(M_{\mathfrak{m}})\to (M_{\mathfrak{m}}+f_{\mathfrak{m}}(M_{\mathfrak{m}}))/f_\mathfrak{m}(M_\mathfrak{m})\to 0.\]
As  $M_{\mathfrak m}$ is automorphism invariant, we deduce that $(M_{\mathfrak{m}}+f_{\mathfrak{m}}(M_{\mathfrak{m}}))/f_\mathfrak{m}(M_\mathfrak{m})=0$. Therefore, $(M+f(M)/f(M))_{\mathfrak m}=0$ for every maximal ideal $\mathfrak m$; and this means that $(M+f(M))/f(M)=0$. Thus, $M$ is invariant under $f$.
\end{proof}

\noindent Bearing in mind that a module $M$ is quasi-injective if and only if it is invariant under any endomorphism of its injective envelope, the above arguments can be easily adapted to prove:

\begin{prop}\label{local-qi}
Let $R$ be a commutative noetherian ring and $M,$ an $R$-module. The following are equivalent.
\begin{enumerate}
\item $M$ is quasi-injective.
\item $M_{\mathfrak p}$ is a quasi-injective $R_{\mathfrak p}$-module for every prime ideal ${\mathfrak p}$.
\item $M_{\mathfrak m}$ is a quasi-injective $R_{\mathfrak m}$-module for every maximal ideal ${\mathfrak m}$.
\end{enumerate}
\end{prop}

\noindent We can now extend Corollary \ref{comm-nonF2} and Corollary \ref{local} and prove the following.

\begin{thm} \label{comm-noetherian}
Let $R$ be a commutative noetherian ring. Then any automorphism-invariant $R$-module is quasi-injective.
\end{thm}

\begin{proof}
Let $M$ be an automorphism-invariant $R$-module. Then $M_{\mathfrak p}$ is an automorphism-invariant  $R_{\mathfrak p}$-module for every prime ideal $\mathfrak p$ by Proposition~\ref{local-ai}. As each $R_{\mathfrak p}$ is local, we deduce from Corollary~\ref{local} that $M_{\mathfrak p}$ is quasi-injective for every prime ideal $\mathfrak p$. And therefore, $M$ is quasi-injective by Proposition~\ref{local-qi}.
\end{proof}

\begin{rem} \rm
In \cite{BCS} it was asked whether in the case of abelian groups, the notions of automorphism-invariance and quasi-injectivity coincide. Our Theorem \ref{comm-noetherian} answers this open question in the affirmative. 
\end{rem}

\begin{rems} 
\begin{enumerate}
\item Teply constructed in \cite{Teply}  an example of a non-finitely generated automorphism-invariant module over the commutative ring $\mathbb{Z}[x_1, x_2, \ldots, x_n, \ldots]$ which is not quasi-injective. Therefore, we cannot expect to extend our previous theorem to  modules over non-noetherian commutative rings.

\item On the other hand, the ring $R$ of all eventually constant sequences of elements of the field $\mathbb{F}_2$ is a commutative boolean ring which is automorphism-invariant as a module over itself, but it is not self-injective (see \cite{ESS}). This shows that we cannot drop the noetherian condition from the hypotheses of the above theorem, even if we assume that the ring is boolean and the module is finitely generated and non-singular.

\item It is straightforward to check that being automorphism-invariant is a Morita-invariant property. Therefore, the statement in Theorem~\ref{comm-noetherian} is also valid if we
replace ``commutative noetherian ring $R$" by ``a full  matrix ring $\mathbb{M}_n(R)$ over a commutative noetherian ring $R$". 

\item In \cite[p. 362]{JS}, an example of a module over an HNP ring is given claiming that this module is pseudo-injective (equivalently, automorphism-invariant) but not quasi-injective. We would like to point out that this example is incorrect. Let us choose as the field $\Phi$ in that example the field $\mathbb C$ of complex numbers.  Then $B$ is a $\mathbb C$-algebra and thus, $\End_B(M)$ is also a $\mathbb C$-algebra. Therefore, $\End(M)$ cannot have a homomorphic image isomorphic to $\mathbb F_2$. This means that $M$ must be quasi-injective by Theorem~\ref{noF2} (see also \cite[Theorem~3]{GS2}), thus contradicting the assertion in \cite[p. 362]{JS} that $M$ is not quasi-injective. Since this alleged counter-example is not correct, we would like to pose the question: whether every torsion automorphism-invariant module (or even any automorphism-invariant module) over an HNP ring is quasi-injective.
\end{enumerate}

\end{rems}

\section{When is an $\mathcal X$-automorphism coinvariant module $\mathcal X$-endomorphism coinvariant.}

\noindent We finish this paper by dualizing the results obtained in the previous sections for $\mathcal X$-automorphism coinvariant modules. Clearly, any $\mathcal X$-endomorphism coinvariant module is $\mathcal X$-automorphism coinvariant. We give below an example that shows that the converse does not need to hold.

Following the convention in the previous sections for injective modules, when $\mathcal X$ is the class of projective modules, we will call $\mathcal X$-automorphism coinvariant (resp. $\mathcal X$-endomorphism coinvariant) modules just automorphism-coinvariant (resp., quasi-projective) modules. Recently it has been shown in \cite{GKBS} that over a right perfect ring, a module $M$ is automorphism-coinvariant if and only if for every submodule $N$ of $M$, any epimorphism $\varphi: M\rightarrow M/N$ can be lifted to an endomorphism of $M$.

\begin{ex} \label{no-qp}
Let $R$ be the ring given in Example~\ref{noqi} and $M=e_{11}R$. As $R$ is a finite-dimensional algebra over $\mathbb{F}_2$, the functors
$$\Hom_{\mathbb{F}_2}(-,\mathbb{F}_2):\ModR \rightarrow \RMod$$
and
$$\Hom_{\mathbb{F}_2}(-,\mathbb{F}_2):\RMod \rightarrow \ModR$$
establish a contravariant equivalence between the subcategories of left and right finitely generated modules over $R$. Moreover, as $M$ is a finitely generated right $R$-module, its injective envelope $E(M)$ is also finitely generated.

Therefore, $\Hom_{\mathbb{F}_2}(E(M),\mathbb{F}_2)$ is the projective cover of $\Hom_{\mathbb{F}_2}(M,\mathbb{F}_2)$ and we deduce that $\Hom_{\mathbb{F}_2}(M,\mathbb{F}_2)$ is an automorphism-coinvariant left $R$-module.
Moreover, it may be noticed that $\Hom_{\mathbb{F}_2}(M,\mathbb{F}_2)$ is not invariant under endomorphisms of its projective cover
$\Hom_{\mathbb{F}_2}(E(M),\mathbb{F}_2)$ because otherwise $M\cong \Hom_{\mathbb{F}_2}(\Hom_{\mathbb{F}_2}(M,\mathbb{F}_2),\mathbb{F}_2)$ would be invariant under endomorphisms of $E(M)\cong \Hom_{\mathbb{F}_2}(\Hom_{\mathbb{F}_2}(M,\mathbb{F}_2),\mathbb{F}_2)$. Therefore, $\Hom_{\mathbb{F}_2}(M,\mathbb{F}_2)$ is not quasi-projective.
\end{ex}

Using similar arguments as in Lemma \ref{indecomp-1}, Proposition \ref{indecomp} and Theorem \ref{noei-struct}, we can prove.

\begin{lem} \label{dual-indecomp-1}
Let $M$ be an $\mathcal X$-automorphism coinvariant module with an epimorphic $\mathcal X$-cover $p:X\rightarrow M$ such that $\End(X)/J(\End(X))$ is von Neumann regular right self-injective and idempotents lift modulo $J(\End(X))$. If $X$ is indecomposable, then $M$ is $\mathcal X$-endomorphism coinvariant.
\end{lem}

\begin{prop} \label{dual-indecomp}
Let $M$ be an indecomposable $\mathcal X$-automorphism coinvariant module with an epimorphic $\mathcal X$-cover $p:X\rightarrow M$ such that $\End(X)/J(\End(X))$ is von Neumann regular right self-injective and idempotents lift modulo $J(\End(X))$. Then the following statements are equivalent:
\begin{enumerate}
\item $M$ is $\mathcal X$-endomorphism coinvariant.
\item $X$ is indecomposable.
\end{enumerate}
\end{prop}

\begin{thm} \label{dual-noei-struct}
Let $M$ be an indecomposable $\mathcal X$-automorphism coinvariant module with an epimorphic $\mathcal X$-cover $p:X\rightarrow M$ such that $\End(X)/J(\End(X))$ is von Neumann regular right self-injective and idempotents lift modulo $J(\End(X))$. Assume that $M$ is not $\mathcal X$-endomorphism coinvariant. Then $\End(M)/J(\End(M))\cong \mathbb F_2$ and $\End(X)/J(\End(X))$ has a homomorphic image isomorphic to $\mathbb F_2 \times \mathbb F_2$.
\end{thm}

We are now going to apply the results above obtained to the special case in which $\mathcal X$ is the class of projective modules. Our first result is a dual of Theorem \ref{thm:main}.

\begin{thm}\label{dual-thm:main} Let $R$ be a right perfect ring and $M$, an automorphism-coinvariant right $R$-module with a projective cover $\pi: P\to M$ such that $M$ is not quasi-projective. Assume that $M$ is indecomposable with finite dual Goldie dimension. Then
\begin{enumerate}
\item $\End(M)/J(\End(M))\cong \mathbb{F}_2.$
\item There exists an $n\geq 2$ and a set $\{P_i\}_{i=1}^n$ of non-isomorphic indecomposable projective modules such that $P=\oplus_{i=1}^n P_i$ and $\End_R(P_i)/J(\End(P_i))\cong \mathbb{F}_2$ for every $i=1,\ldots,n$.
\end{enumerate}
\end{thm}

\begin{proof}
By Lemma \ref{dual-noei-struct}, $\End(M)/J(\End(M))\cong \mathbb{F}_2$.
By Theorem~\ref{struct}, we know that $M=N\oplus L$, where $N$ is a semiboolean module and $L$, a quasi-projective module. As we are assuming that $M$ is indecomposable and non quasi-projective, we deduce that $M=N$ is a semiboolean module. On the other hand, as $R$ is right perfect, the projective cover $P$ of $M$ is a direct sum of indecomposable direct summands. Moreover, as $M$ has finite dual Goldie dimension and $M$ is square-free (because it is semiboolean and idempotents in $\End_R(M)/J(\End_R(M))$ lift to idempotents of $\End(M)$ \cite{GKS}), we deduce that there exists a finite set of non-isomorphic indecomposable projective modules $\{P_i\}_{i=1}^n$ such that $P=\oplus_{i=1}^n P_i$. Now, $\End_R(P)/J(\End_R(P))\cong \sqcap_{i=1}^n \End_R(P_i)/J(\End_R(P_i))= \sqcap_{i=1}^n D_i$, where each $D_i=\End_R(P_i)/J(\End_R(P_i)$ is a division ring.
By the proof of Theorem~\ref{struct}, we have that $P$ is also a semiboolean module and this means that any element in $\End_R(P)/J(\End_R(P))$ is idempotent and therefore, each $D_i\cong \mathbb{F}_2$. Finally, $n\geq 2$ by Theorem~\ref{dual-noei-struct}.
\end{proof}

We can now adapt the arguments in Corollary~\ref{local} to first show:

\begin{cor}\label{local-co}
Any automorphism-coinvariant module over a commutative local perfect ring is quasi-projective.
\end{cor}

This extends to any commutative perfect ring as shown below. 

\begin{thm}\label{dual-perfect}
Let $R$ be a commutative perfect ring. Then any automorphism-coinvariant module over $R$ is quasi-projective.
\end{thm}

\begin{proof}
Let $M$ be an automorphism-coinvariant module over $R$. Being a commutative (semi)perfect ring, $R$ is a finite direct product of local rings, say $R=\sqcap_{i=1}^n R_i$. Then $M$ is a direct product $M=\sqcap_{i=1}^n M_i$, where each $M_i$ is an automorphism-coinvariant $R_i$-module. Therefore, each $M_i$ is a quasi-projective $R_i$-module by the above corollary. And thus, $M=\sqcap_{i=1}^n M_i$ is a quasi-projective $\sqcap_{i=1}^n R_i$-module.
\end{proof}

\begin{rem}
Observe that the same arguments used in the proofs show that the statements of Theorem~\ref{dual-thm:main} and Theorem~\ref{dual-perfect} are also valid if we replace ``(right) perfect ring" and ``(right) module" by ``semiperfect ring" and ``finitely generated (right) module", respectively.

Let us finally note that Example~\ref{no-qp} shows that over a noncommutative perfect ring, an automorphism-coinvariant module does not need to be quasi-projective.
\end{rem}

\bigskip

{\bf Acknowledgment.} We would like to thank the referee for several suggestions that improved an earlier version of Theorem~\ref{comm-noetherian}.

\end{document}